\documentclass{amsart}
\usepackage{amssymb} 
\usepackage{amsmath} 
\usepackage{amscd}
\usepackage{amsbsy}
\usepackage{enumerate}
\usepackage[matrix,arrow]{xy}
\usepackage[hidelinks]{hyperref}
\usepackage{mathrsfs}
\usepackage{color}
\usepackage{mathtools,caption}
\usepackage{longtable}
\usepackage{comment}

\DeclareMathOperator{\norm}{Norm}

\newcommand{\Q}{\mathbb Q}
\newcommand{\F}{\mathbb F}
\newcommand{\Z}{\mathbb Z}
\newcommand{\OK}{\mathcal O_K}
\newcommand{\fa}{\mathfrak a}
\newcommand{\fb}{\mathfrak b}
\newcommand{\fy}{\mathfrak y}

\newcommand{\fp}{\mathfrak p}
\newcommand{\fq}{\mathfrak q}

\newcommand{\tu}{\tilde u}

\newcommand{\OL}{\mathcal O_L}

\newtheorem{thm}{Theorem}
\newtheorem{lem}{Lemma}[section]
\newtheorem{prop}[lem]{Proposition}

\theoremstyle{definition}

\theoremstyle{remark}

\begin{document}

\title[]{
A Lucas-Lehmer approach 
to generalised 
Lebesgue-Ramanujan-Nagell equations 
}

\author{Vandita Patel}
\address{School of Mathematics, University of Manchester, Oxford Road, Manchester M13 9PL, United Kingdom
}
\email{vandita.patel@manchester.ac.uk}

\date{\today}


\begin{abstract}
We describe a computationally efficient approach to resolving equations of the form
$C_1x^2 + C_2 = y^n$ in coprime integers, 
for fixed values of $C_1$, $C_2$ subject to further conditions.  
We make use of a
factorisation argument and the Primitive Divisor Theorem 
due to Bilu, Hanrot and Voutier.
\end{abstract}

\maketitle

\section{Introduction} \label{intro}

Ramanujan \cite{Ramanujan}, in 1913, 
conjectured that the
 only positive integral solutions 
to the equation 
\[
x^2 + 7 = 2^n
\]
are
\[
(1,3),\quad  (3,4),\quad (5,5),\quad (11,7),\quad (181,15).
\]
This was proven by Nagell \cite{Nagell} in 1948, and 
the equation is now called the Ramanujan--Nagell equation. More generally,
equations of the form 
\begin{equation}\label{eqn:genRN}
C_1 x^2+C_2= C_3^n
\end{equation}
where $C_1$, $C_2$, $C_3$ are fixed non-zero integers are referred to 
as generalised Ramanujan--Nagell equations. Various special cases of \eqref{eqn:genRN}
have been considered by many authors using a variety of methods \cite{BPsurvey}.
For any such $C_1$, $C_2$, $C_3$, it is straightforward to reduce \eqref{eqn:genRN} 
 to solving $S$-unit equations. This allows 
us to conclude that the set of solutions to \eqref{eqn:genRN}
is finite
by a famous theorem of Siegel. It also gives
an effective algorithm for solving the equation.

\medskip

In this paper we consider the generalisation
\begin{equation}\label{eqn:LN}
C_1 x^2+C_2=y^n
\end{equation}
where $C_1$, $C_2$ are fixed, but $x$, $\lvert y \rvert > 1$, $n \ge 3$ are unknown. 
Here Baker's theory gives astronomical bounds
on the size of the solutions $(x,y,n)$, but does not alone
give a practical method for determining them. 
In fact, the earliest special case of \eqref{eqn:LN}
appears to be due to Victor Lebesgue \cite{Lebesgue} who in 1850 solved \eqref{eqn:LN}
for $C_1=C_2=1$. In 1948, Nagell \cite{Nagell} solved the cases $C_1=1$, $C_2=3$, $5$,
and it is now usual to refer to the equation 
\begin{equation}\label{eqn:Cohn}
x^2+C=y^n
\end{equation}
as the Lebesgue--Nagell equation.  In a series of papers (culminating in \cite{Cohn1}),
J.H.E.\ Cohn solved \eqref{eqn:Cohn} for many values of $C>0$.
After the appearance of the celebrated theorem of Bilu, Hanrot
and Voutier (BHV) on primitive divisors of Lucas and Lehmer sequences \cite{BHV},
Cohn revisited \eqref{eqn:Cohn} in \cite{Cohn2},
showing that BHV allows for an easy resolution
for $77$ values in the range $1 \le C \le 100$. 
The cases $C = 74$ and $C=86$ were solved by Mignotte and de Weger \cite{MignotteWeger}. Using the modular approach based on 
Galois representations of elliptic curves and modular forms, the cases $C=55$ and $C=95$ were solved by 
Bennett and Skinner \cite{BenS}.
The remaining $19$ values were dealt with in a pioneering
paper due Bugeaud, Mignotte and Siksek \cite{BMS}, 
which combines Baker's theory with the modular approach.
A related work which relies heavily on BHV
is due to Abu Muriefah, Luca, Siksek and Tengely \cite{ALST},
and adapts Cohn's method to the equation $x^2+C=2y^n$ (see also
\cite{Tengely}, \cite{Tengely2} for related equations).

In view of Cohn's work, it is natural to consider \eqref{eqn:LN},
which we refer to as the generalised Lebesgue--Ramanujan--Nagell equation.
We extend Cohn's method so that it applies in far greater
generality. 

More precisely, we study equations of the form:
\begin{equation}\label{eqn:main1}
C_1x^2 + C_2 = y^n, \quad x,~y \in \Z^{+}, \quad \gcd(C_1x^2,C_2,y^n)=1, \quad n \ge 3.
\end{equation}
We may assume without loss of generality that $n$ is an odd prime,
or that $n=4$.
We prove the following.
\begin{thm}\label{thm:main}
Let $C_1$ be a positive squarefree integer and $C_2$ a positive integer. Write 
$C_1C_2 = cd^2$ where $c$ is squarefree. We assume that $C_1C_2 \not\equiv 7
\pmod{8}$.  Let $p$ be an odd prime for which the equation 
\begin{equation}\label{eqn:main}
C_1x^2 + C_2 = y^p, \quad x,~y \in \Z^{+}, \quad \gcd(C_1x^2,C_2,y^p)=1, 
\end{equation}
has a solution $(x,y)$.
Then either, 
\begin{enumerate}[(i)]
\item $p \le 5$, or
\item $p=7$ and $y=3$, $5$ or $9$, or 
\item $p$ divides the class number of $\Q(\sqrt{-c})$, or
\item $p \mid \left( q - \left(\frac{-c}{q}\right)\right)$, where $q$ is some
prime $q\mid d$ and $q\nmid 2c$.  
\end{enumerate}
\end{thm}

In Section~\ref{sec:comp}, we give an effective method that solves \eqref{eqn:main1}
for a given value of $n \ge 3$. Our algorithm relies on standard algorithms for
solving Thue equations and determining integral points on elliptic curves.
We implemented our method in \texttt{Magma} \cite{Magma} which has inbuilt
implementation of these algorithms and we used this, together with Theorem~\ref{thm:main}, to determine
the solutions to \eqref{eqn:main1} for $2 \le C_1 \le 10$, $1 \le C_2 \le 80$
subject to the restrictions: $C_1$ is squarefree, $\gcd(C_1,C_2)=1$, and  
$C_1C_2 \not \equiv 7 \pmod{8}$. Our results are given in Section~\ref{sec:solns}. We point out that the case $C_1=1$ and $1 \le C_2 \le 100$ is 
completely solved in \cite{BMS}, which incorporates the earlier
work of Cohn, Bennett and Skinner, and Mignotte and de Weger.

%
%
The author thanks Yann Bugeaud and Szabolcs  Tengely for useful conversations.
%

%

%

\section{Primitive prime divisors of 
Lehmer sequences} {\label{sec:BHV}}

A \textit{Lehmer pair} is a pair of algebraic integers $\alpha,\beta$,  such that $(\alpha +\beta)^2$ and $\alpha\beta$ are non--zero coprime rational integers and $\alpha/\beta$ is not a root of unity. The \textit{Lehmer sequence} 
associated to the Lehmer pair $(\alpha,\beta)$ is
\begin{equation*}
\tu_n=\tu_n(\alpha,\beta)=\begin{cases}
\frac{\alpha^n-\beta^n}{\alpha-\beta}, & \text{if $n$ is odd},\\
\frac{\alpha^n-\beta^n}{\alpha^2-\beta^2}, & \text{if $n$ is even}.
\end{cases}
\end{equation*}
A prime $p$ is called a \textit{primitive divisor} of $\tu_n$ if
it divides $\tu_n$ but does not divide 
$(\alpha^2-\beta^2)^2\cdot\tu_1\cdots\tu_{n-1}$. 
We shall make use of the following celebrated theorem~\cite{BHV}.
\begin{thm}[Bilu, Hanrot and Voutier]\label{thm:non_defective}
Let $\alpha$, $\beta$ be a Lehmer pair. Then
$\tu_n(\alpha,\beta)$ has a primitive divisor for all $n>30$,
and for all prime $n>13$.
\end{thm}

A Lehmer pair $(\alpha,\beta)$ is called $n$-\textbf{defective}
if $\tu_n(\alpha,\beta)$ does not have a primitive divisor.
Two Lehmer pairs $(\alpha,\beta)$ and $(\alpha^\prime,\beta^\prime)$
are said to be \textbf{equivalent} if $\alpha/\alpha^\prime=\beta/\beta^\prime
\in \{\pm 1, \pm \sqrt{-1}\}$. Table 2 of \cite{BHV} gives
all equivalence classes of $n$-\textbf{defective} Lehmer pairs for 
all $6<n \le 30$ except $n \ne 8$, $10$, $12$. In particular,
\begin{itemize}
\item there are no $11$-defective Lehmer pairs; 
\item every $13$-defective
Lehmer pair is equivalent to 
$((\sqrt{a}+\sqrt{b})/2,(\sqrt{a}-\sqrt{b})/2)$ where $(a,b)=(1,-7)$; 
\item every $7$-defective Lehmer pair is equivalent to 
$((\sqrt{a}+\sqrt{b})/2,(\sqrt{a}-\sqrt{b})/2)$ where 
$(a,b)=(1,-7)$, $(1,-19)$, $(3,-5)$, $(5,-7)$, $(13,-3)$, $(14,-22)$.
\end{itemize}

\section{Preliminary Descent} \label{sec:descent}

Throughout Sections~\ref{sec:descent} and~\ref{sec:Lehmer}
we maintain the following assumptions and notation:
\begin{enumerate}[(a)]
\item $C_1 $ is a squarefree positive integer, $C_2$ is
a positive integer and $\gcd(C_1,C_2)=1$. We moreover
suppose that $C_1 C_2 \not\equiv 7 \pmod{8}$. We write
$C_1 C_2=c d^2$ where $c$, $d$ are positive integers and $c$ is squarefree.
\item $(x,y)$ satisfies \eqref{eqn:main}.
\item $p$ is an odd prime. Moreover, if $p=3$ then
we suppose additionally that 
$C_1 C_2/3$ is not a square.
\item $p$ does not divide the class number of $\Q(\sqrt{-c})$.
\end{enumerate}

\begin{lem}\label{lem:fact}
Let $(x,y)$ be a solution to 
\eqref{eqn:main}. Let $\OK$ be the ring of integers for the number field 
$K = \Q(\sqrt{-c})$.
Then there is some $\delta \in \OK$ such that
\begin{equation}\label{eqn:descent}
C_1 x + d \sqrt{-c} = \frac{\delta^p}{C_1^{(p-1)/2}}.
\end{equation}

Moreover, we have
\begin{equation}\label{eqn:subconj}
\frac{\delta^p}{C_1^{p/2}} - \frac{{\bar\delta}^p}{C_1^{p/2}} = 2d\cdot \frac{\sqrt{-c}}{\sqrt{C_1}}.
\end{equation}
\end{lem}

\begin{proof}
Let $K=\Q(\sqrt{-c})$ and $\OK$ its ring of integers.
Let $h_K$ be the class number of $K$ and we assume that $p\nmid h_K$. 
As $C_1 C_2 \not\equiv 7 \pmod{8}$ we have that $y$ is odd.

As $C_1$, $c$ are both squarefree, $\gcd(C_1,C_2)=1$
 and
$C_1 C_2=c d^2$ it follows that $C_1 \mid c$.
Let $C_1 = p_1\cdots p_r$ 
where we note that the primes $p_1,\ldots p_r$ ramify in $K$.

We factorise equation~\eqref{eqn:main} in $\OK$ as follows:

\[
\left(C_1 x + d\sqrt{-c}\right)\left(C_1 x-d\sqrt{-c}\right) = C_1\cdot y^p = p_1\cdots p_r\cdot y^p.
\]

Let us write $\fp_i$ for the prime ideal above $p_i$ where $1\leq i \leq r$. 
Let $\fa=\fp_1 \cdots\fp_r$ and we obtain:
\begin{align*}
(C_1 x+ d\sqrt{-c})\OK  
 &=  \fp_1\cdots \fp_r\cdot \fy ^p \\
 &= \fa^{1-p}\cdot(\fa\fy)^p \\
 &= (C_1^{(1-p)/2})\cdot(\fa\fy)^p
\end{align*}
where $\fa \fy$ is a principal ideal of $\OK$. Indeed,  $[\fa\fy]^p = 1$ in the
class group. Therefore the class $[\fa\fy]$ has order dividing $p$. 
By assumption $p \nmid h_K$. Thus $\fa \fy$ is principal.

Therefore, we write $\fa\fy = \delta\OK$. The unit group of $\OK$
has order $2$, $4$ or $6$, and is therefore $p$-divisible,
unless $p=3$. However, for $p=3$ we have assumed
that $C_1 C_2/3$ is a non-square and therefore $K \ne \Q(\sqrt{-3})$,
and so the order of the unit group is $2$ or $4$. Thus
in all cases the unit group is $p$-divisible. Thus adjusting
$\delta$ by an appropriate unit we obtain \eqref{eqn:descent}.
Subtracting the conjugate from \eqref{eqn:descent}, we get
\begin{equation*}
\frac{\delta^p}{C_1^{(p-1)/2}} - \frac{{\bar\delta}^p}{C_1^{(p-1)/2}} = 2d\sqrt{-c},
\end{equation*}
which is equivalent to \eqref{eqn:subconj}.
This completes the proof of the lemma.
\end{proof}

\noindent \textbf{Remark.}
If $C_1 C_2 \equiv 7 \pmod{8}$ then it is possible
for $y$ to be even. In that case it is no longer
true that we can express $(C_1 x+ d \sqrt{-c}) \OK$
in the form $\fa \fy^p$ where $\fa^2=C_1 \OK$.

\section{Satisfying the Lehmer condition} \label{sec:Lehmer}

Let $K = \Q(\sqrt{-c})$ as before, and consider the extension, $L/K$, where $L=\Q(\sqrt{-c}, \sqrt{C_1})$. Observe that $L/K$ is trivial if $C_1=1$,
and is quadratic otherwise.
We write $\OL$ for its ring of integers and set
$
\alpha = \delta/\sqrt{C_1} ,\;  \beta= \bar{\delta}/\sqrt{C_1}.
$
Thus equation~\eqref{eqn:subconj} becomes
\begin{equation}{\label{eqn:alphabeta}}
\alpha^p - \beta^p = 
2d\cdot \frac{\sqrt{-c}}{\sqrt{C_1}}. 
\end{equation}

For the remainder of this section, 
in the case $-c \not\equiv 1 \pmod{4}$ we let
\begin{equation}\label{eqn:delta1}
\delta = r + s\sqrt{-c}, \quad \bar\delta = r - s\sqrt{-c}, 
\end{equation}
where $r$, $s$ are integers. 
In the case $-c \equiv 1 \pmod{4}$ we let
\begin{equation}\label{eqn:delta2}
\delta = \frac{r + s\sqrt{-c}}{2}, \quad \bar\delta = \frac{r - s\sqrt{-c}}{2},
\end{equation}
where $r$ and $s$ are either both odd or both even.

\begin{lem}\label{lem:algint}
$\alpha$ and $\beta$ are algebraic integers. Moreover,
\[
\alpha \beta=y, \qquad \sqrt{C_1} x+ \sqrt{-C_2} = \alpha^p,
\qquad \sqrt{C_1} x- \sqrt{-C_2}=\beta^p.
\]
\end{lem}
\begin{proof}
By the proof of Lemma~\ref{lem:fact}, 
$\fa^2= C_1 \OK$ and so
$\sqrt{C_1} \OL=\fa \OL$ which divides $\fa \fy \OL=\delta \OL$.
Hence $\alpha=\delta/\sqrt{C_1}$ is an algebraic integer.

Dividing \eqref{eqn:descent} by $\sqrt{C_1}$ gives
$\sqrt{C_1} x+ \sqrt{-C_2} = \alpha^p$ and applying
complex conjugation gives
$\sqrt{C_1} x- \sqrt{-C_2} = \beta^p$.
Multiplying the two equations gives $y^p=(\alpha \beta)^p$.
But as $\alpha$, $\beta$ are complex conjugates, $y$, $\alpha \beta$
are both positive, so $y=\alpha \beta$ as required.
\end{proof}

\begin{lem}\label{lem:alphabetasquare}
$(\alpha+\beta)^2$ is a non--zero rational integer.
\end{lem}
\begin{proof}
By Lemma~\ref{lem:algint},
$(\alpha+\beta)^2$ is an algebraic integer. However, 
\[
(\alpha+\beta)^2 = \left(\frac{\delta + \bar\delta}{\sqrt{C_1}}\right)^2 = 
\begin{cases}
4r^2/C_1 &\text{ if } -c \not\equiv 1 \pmod{4}\\
r^2/C_1  &\text{ if } -c \equiv 1 \pmod{4},
\end{cases}
\]
thus $(\alpha+\beta)^2$ is a rational number as well as being
an algebraic integer. Thus it is a rational integer.

Next, we suppose that $(\alpha + \beta)^2 = 0$. 
Then $\delta$ is purely imaginary, and \eqref{eqn:descent}
implies that $x=0$. This contradicts our assumption
that $x$ is positive.
\end{proof}

The following is immediate from Lemma~\ref{lem:algint}.
\begin{lem}\label{lem:alphabeta}
$\alpha\beta$ is a non-zero rational integer.
\end{lem}

\begin{lem}\label{lem:coprime}
$(\alpha+\beta)^2$ and $\alpha\beta$ are coprime.
Moreover $\alpha/\beta$ is not a unit.
\end{lem}
\begin{proof}
Suppose that $(\alpha + \beta)^2$ and $\alpha\beta$ are not coprime. Then there
exists a prime $\fq$ of $\OL$ which divides both. Thus, 
$\fq \mid \alpha, \beta$. By Lemma~\ref{lem:algint},
$\fq \mid y$ and $\fq \mid (2 \sqrt{C_1} x)$. As we saw previously,
$y$ must be odd. Hence $\fq \mid y$ and $\fq \mid C_1 x^2$,
contradicting our coprimality assumption.

Finally suppose $\alpha/\beta$ is a unit. In particular $\alpha \mid \beta$
and $\beta \mid \alpha$. We claim that $\alpha$ is a unit.
Suppose otherwise, and let $\fq \mid \alpha$ be a prime of $\OL$.
Then $\fq \mid \beta$ and we obtain a contradiction as above. Hence
$\alpha$ must be a unit and so $\beta$ is a unit. Therefore
$y=\alpha \beta$ is a unit in $\Z$. Thus $y = \pm 1$. This contradicts
$C_1 x^2+C_2=y^p$ and the positivity assumption for the solution.
\end{proof}

Lemmata~\ref{lem:algint}, \ref{lem:alphabetasquare},  \ref{lem:alphabeta},
\ref{lem:coprime} provide a proof to the following:
\begin{prop}{\label{prop:Lehmercond}}
Let $\alpha,\beta$ be as above. Then $\alpha$ and $\beta$ are algebraic integers. Moreover, $(\alpha+\beta)^2$ and $\alpha\beta$ are non--zero, coprime, rational integers and $\alpha/\beta$ is not a unit.
\end{prop}

\section{Proof of Theorem~\ref{thm:main}}
In this section we prove Theorem~\ref{thm:main}.
We suppose $p>5$ and $p \nmid h_K$. We would like
to show that $(p,y)=(7,3)$, $(7,5)$, $(7,9)$ or there is some prime
$q \mid d$, $q \nmid 2c$ such that $p \mid B_q$ where
\[
B_q=\begin{cases}
q-1 & \text{if $\left(\frac{-c}{q}\right)=1$}\\
q+1 & \text{if $\left(\frac{-c}{q}\right)=-1$}.
\end{cases}
\]
Let $(\alpha,\beta)$ be as above.
Proposition~\ref{prop:Lehmercond} tells us that $(\alpha,\beta)$ is 
indeed a Lehmer pair. We denote by $\tu_k$ the associated Lehmer sequence. 
From \eqref{eqn:delta1}, \eqref{eqn:delta2} we have
\begin{equation}\label{eqn:diff}
\alpha-\beta=
\begin{cases}
\frac{2s\sqrt{-c}}{\sqrt{C_1}} & \text{if $-c \not\equiv 1 \pmod{4}$}\\
\frac{s\sqrt{-c}}{\sqrt{C_1}} & \text{if $-c \equiv 1 \pmod{4}$}.
\end{cases}
\end{equation}
Combining with \eqref{eqn:alphabeta} gives
\begin{equation}\label{eqn:dd}
\tu_p=\frac{\alpha^p-\beta^p}{\alpha-\beta}=
\begin{cases}
\frac{d}{s} & \text{if $-c \not\equiv 1 \pmod{4}$}\\
\frac{2d}{s} & \text{if $-c \equiv 1 \pmod{4}$}.
\end{cases}
\end{equation}
We suppose first that $(\alpha,\beta)$ is not $p$-defective.
Thus there is a prime
$q \mid \tu_p$ such that $q \nmid (\alpha^2-\beta^2)^2$
and $q \nmid \tu_1 \tu_2 \cdots \tu_{p-1}$.
We claim that $q \ne 2$. Suppose $q=2$. 
Let $\fq$ be a prime of $\OL$ dividing $q$.
Then
\[
\alpha^p \equiv \beta^p \pmod{\fq}, \qquad \alpha \not\equiv \beta \pmod{\fq}.
\]
Hence $\alpha/\beta$ has order $p$ in $(\OL/\fq)^*$. This 
group has order $\norm(\fq) -1$. As $L$ has degree $4$, 
$\norm(\fq)=2$ or $4$ or $16$. Thus $p=3$ or $5$ which 
contradicts $p>5$. Therefore $q \ne 2$.

Next we claim that $q \nmid C_1$. Suppose $q \mid C_1$.
Let $\fq$ be a prime of $\OL$ dividing $q$.
Then $\alpha^p \equiv \beta^p \pmod{\fq}$ and $\sqrt{C_1} \equiv 0
\pmod{\fq}$. By Lemma~\ref{lem:algint}, 
$\fq \mid 2 \sqrt{-C_2}$. Hence $q \mid C_1$ and $q \mid (2 C_2)$.
But $C_1$, $C_2$ are coprime and $q \ne 2$ giving a contradiction.
Thus $q \nmid C_1$.

From \eqref{eqn:diff}, the fact that $q \nmid C_1$ and
$q \nmid (\alpha^2-\beta^2)^2$ we deduce that $q \nmid c$
as required.

 Let $\fq$ be a prime of $K$
above $q$. Then $\delta/\overline{\delta} \not\equiv 1 \pmod{\fq}$
and $(\delta/\overline{\delta})^p \equiv 1 \pmod{\fq}$.
If $(-c/q)=1$ then $\F_\fq=\F_q$ and so $p \mid (q-1)$.
If $(-c/q)=-1$ then $\F_\fq=\F_{q^2}$. However, 
$\delta/\overline{\delta} \pmod{\fq}$ belongs to the kernel
of the norm map $\F_{q^2}^* \rightarrow \F_q^*$ which
has order $q+1$. Thus in this case, $p \mid (q+1)$.
Hence $p \mid B_q$.

To complete the proof we need to consider the case where 
$(\alpha,\beta)$ is $p$-defective. By Theorem~\ref{thm:non_defective}
and the discussion following it, we know that $p=7$ or $13$.
Moreover $(\alpha,\beta)$ is equivalent to $(\alpha^\prime,\beta^\prime)
=((\sqrt{a}+\sqrt{b})/2, (\sqrt{a}-\sqrt{b})/2)$ where the possibilities
for $(a,b)$ are listed in that discussion. Recall
$\alpha/\alpha^\prime=\beta/\beta^\prime \in \{\pm 1, \pm \sqrt{-1}\}$.
Moreover, $y=\alpha\beta$. Thus if $\alpha/\alpha^\prime=\beta/\beta^\prime=\pm \sqrt{-1}$
we obtain $y=-\alpha^\prime\beta^\prime$. However, $y$ is positive and $\alpha^\prime \beta^\prime$
is also positive in all cases. Thus $\alpha/\alpha^\prime=\beta/\beta^\prime=\pm 1$.
Hence $y=\alpha^\prime \beta^\prime=(a-b)/4$. When $(a,b)=(1,-7)$, $(13,-3)$, $(3,-5)$,
we have $y=2$, $4$, $2$ respectively. This contradicts our assumption that
$C_1 C_2 \not \equiv 7 \pmod{8}$. We are reduced to the case where $p=7$,
and $(a,b)=(1,-19)$, $(5,-7)$, $(14,-22)$, which respectively give
$y=5$, $3$, $9$. This completes the proof.

\bigskip

We note in passing that it is not possible to eliminate the cases
$p=7$, $y=5$, $3$, $9$. For example, for $p=7$, $y=5$, there are
$59893$ possibilities for a triple $(C_1,C_2,x)$ which satisfies
$C_1 x^2+C_2=y^p=5^7$ and all our other restrictions.

\section{Effectively Determining Solutions}\label{sec:comp}
Let $C_1$, $C_2$ satisfy condition (a) of Section~\ref{sec:descent}.
Theorem~\ref{thm:main} gives a list of possible odd prime
exponents $n=p$ for which \eqref{eqn:main1} might have solutions.
As noted in the introduction, we may without loss of generality
suppose that $n=p$ is an odd prime, or that $n=4$.
In this section, we outline a practical method to compute 
these solutions for fixed such value of $n$.
We consider three cases.

\noindent \textbf{Case I:} $n$ is an odd prime $p \nmid h_K$, and if $p=3$
then $C_1 C_2/3$ is not a square.
In this case the conditions (a)--(d) 
of Section~\ref{sec:descent} are all satisfied.
Let $r$, $s$ be as in \eqref{eqn:delta1},~\eqref{eqn:delta2}.
Let
\[
d^\prime=\begin{cases}
d & \text{if $-c \not \equiv 1 \pmod{4}$}\\
2d & \text{if $-c \equiv 1 \pmod{4}$}.
\end{cases}
\]
From \eqref{eqn:dd} we obtain $s \mid d^\prime$. Thus we have
only a few possibilities for $s$. To determine
the solutions we merely have to determine the possible
values of $r$ corresponding to each $s \mid d^\prime$.
We shall write down an explicit polynomial $f_s \in \Z[X]$
whose integer roots contain all the possible values of $r$
corresponding to $s$.

Fix $s \mid d^\prime$. 
If $-c \not \equiv 1 \pmod{4}$, we let
\[
f_s(X)=\frac{ (X+s\sqrt{-c})^p - (X-s\sqrt{-c})^p}{2s \sqrt{-c}}
\; - \; \frac{d \cdot C_1^{(p-1)/2}}{s}.
\]
Clearly $f_s \in \Z[X]$. Moreover,
\[
f_s(r)  =
\frac{ \delta^p - \overline{\delta}^p}{\delta-\overline{\delta}}
\; - \; \frac{d \cdot C_1^{(p-1)/2}}{s}
=0
\]
using \eqref{eqn:subconj} and \eqref{eqn:delta1}.

If $-c \equiv 1 \pmod{4}$, we let
\[
f_s(X)=\frac{ (X+s\sqrt{-c})^p - (X-s\sqrt{-c})^p}{2s \sqrt{-c}}
\; - \; \frac{2^{p} \cdot d \cdot C_1^{(p-1)/2}}{s}.
\]
Again $f_s \in \Z[X]$ and
\[
f_s(r)  =
\frac{ (2\delta)^p - (2\overline{\delta})^p}{2(\delta-\overline{\delta})}
\; - \; \frac{2^p \cdot d \cdot C_1^{(p-1)/2}}{s}
=0
\] 
using \eqref{eqn:subconj} and \eqref{eqn:delta2}.

\bigskip

\noindent \textbf{Case II:} $n$ is an odd prime $p$,
with either $p \mid h_K$, or $p=3$
and $C_1 C_2/3$ is a square. In this case we explain how to
reduce \eqref{eqn:main} to a finite number of Thue equations.
These can be solved using standard methods for Thue equations
such as in \cite{Smart}.
As in the proof of Lemma~\ref{lem:fact}, 
write $C_1=p_1 \cdots p_r$ and let $\fp_i$
be the unique prime ideal of $\OK$ above $p_i$. 
Let $\fa=\fp_1 \cdots \fp_r$.
We have 
\[
(C_1 x+ d\sqrt{-c})\OK  
 =  \fa \cdot \fy^p 
\]
where $\fy$ is an ideal of $\OK$. Let $\fb_1,\dotsc,\fb_h$
be ideals of $\OK$ that form a system of representatives for the class group.
Then, for some $1 \le i \le h=h_K$, we have $\fy \fb_i$ is principal.
Therefore $\fa \fb_i^{-p}$ must be principal. We test 
the ideals $\fa \fb_i^{-p}$ for principality. Fix $i$ such that
$\fa \fb_i^{-p}=\epsilon \OK$ where $\epsilon \in K^*$
and write $\fy\fb_i=\delta \OK$, where $\delta \in \OK$. Then
\begin{equation}\label{eqn:prethue}
C_1 x+ d\sqrt{-c} = \mu \cdot \epsilon \cdot \delta^p
\end{equation}
where $\mu$ is a unit. If $p \ne 3$ or $C_1 C_2/3$
is a non-square, then $\mu$ is a $p$-th power
and we can absorb this in the the $\delta^p$ factor. In this
case we suppose $\mu=1$. Otherwise we also consider
$\mu=1$, $\omega=(-1+\sqrt{-3})/2$ and $\omega^2$.
We write $\delta$ as in \eqref{eqn:delta1}, \eqref{eqn:delta2}
depending on whether $-c \not\equiv 1 \pmod{4}$ or $-c \equiv 1 \pmod{4}$.
We then expand \eqref{eqn:prethue} and equate the coefficients
of $\sqrt{-c}$ and clear denominators to obtain an equation 
of the form
\[
F(r,s)=t
\]
where $t$ is a positive integer, and $F \in \Z[X,Y]$
is a homogeneous polynomial of degree $p \ge 3$. 
This is a Thue equation. In our implemention
we used Magma's inbuilt Thue solver which is an implementation of 
the algorithm in Smart's book \cite[Chapter VII]{Smart},
which is based on linear forms in logarithms. 

\bigskip

\noindent \textbf{Case III:} $n=4$. We write
\[
X=C_1 y^2, \qquad Y=C_1^2 xy,
\]
and note that $(X,Y)$ is now an integral point on the elliptic curve
\[
Y^2=X^3-C_1^2 C_2 X.
\]
We apply Magma's inbuilt function for determining
integral points on elliptic curves which is 
based on linear forms in elliptic logarithms,
as described in Smart's book \cite[Chapter XIII]{Smart}.

\newpage
\section{Solutions} \label{sec:solns}
We are interested in solving \eqref{eqn:main1}
for $2 \le C_1 \le 10$, $1 \le C_2 \le 80$
subject to the restrictions: $C_1$ is squarefree, $\gcd(C_1,C_2)=1$, and  
$C_1C_2 \not \equiv 7 \pmod{8}$. As noted previously, we may without loss
of generality suppose that $n=4$ or that $n=p$ is an odd prime.
For each such pair $(C_1,C_2)$, Theorem~\ref{thm:main} yields
a finite set $S(C_1,C_2)$ of odd primes $p$ for which we need to solve \eqref{eqn:main}.
Thus for each such pair $(C_1,C_2)$ we need only solve \eqref{eqn:main1} for $n \in S(C_1,C_2) \cup \{4\}$, and 
for each such value $n$ we may apply one of the methods explained in Section~\ref{sec:comp}.
We implemented our approach in \texttt{Magma} \cite{Magma}. The results of our computation
are given below.

\medskip

\begin{tabular}{||c|c|c|c|c||}
\hline
$C_1$ & $C_2$ & $x$ & $y$ & $n$\\
\hline\hline
$ 2 $ & $ 1 $ & $ 11 $ & $ 3 $ & $ 5 $ \\ \hline
$ 2 $ & $ 5 $ & $ 13 $ & $ 7 $ & $ 3 $ \\ \hline
$ 2 $ & $ 7 $ & $ 19 $ & $ 9 $ & $ 3 $ \\ \hline
$ 2 $ & $ 13 $ & $ 68 $ & $ 21 $ & $ 3 $ \\ \hline
$ 2 $ & $ 13 $ & $ 41 $ & $ 15 $ & $ 3 $ \\ \hline
$ 2 $ & $ 19 $ & $ 1429 $ & $ 21 $ & $ 5 $ \\ \hline
$ 2 $ & $ 19 $ & $ 33 $ & $ 13 $ & $ 3 $ \\ \hline
$ 2 $ & $ 19 $ & $ 2 $ & $ 3 $ & $ 3 $ \\ \hline
$ 2 $ & $ 23 $ & $ 122 $ & $ 31 $ & $ 3 $ \\ \hline
$ 2 $ & $ 25 $ & $ 1 $ & $ 3 $ & $ 3 $ \\ \hline
$ 2 $ & $ 25 $ & $ 134 $ & $ 33 $ & $ 3 $ \\ \hline
$ 2 $ & $ 27 $ & $ 7 $ & $ 5 $ & $ 3 $ \\ \hline
$ 2 $ & $ 31 $ & $ 5 $ & $ 3 $ & $ 4 $ \\ \hline
$ 2 $ & $ 43 $ & $ 10 $ & $ 3 $ & $ 5 $ \\ \hline
$ 2 $ & $ 47 $ & $ 17 $ & $ 5 $ & $ 4 $ \\ \hline
$ 2 $ & $ 49 $ & $ 4 $ & $ 3 $ & $ 4 $ \\ \hline
$ 2 $ & $ 53 $ & $ 423 $ & $ 71 $ & $ 3 $ \\ \hline
$ 2 $ & $ 53 $ & $ 6 $ & $ 5 $ & $ 3 $ \\ \hline
$ 2 $ & $ 55 $ & $ 441 $ & $ 73 $ & $ 3 $ \\ \hline
$ 2 $ & $ 55 $ & $ 12 $ & $ 7 $ & $ 3 $ \\ \hline
$ 2 $ & $ 73 $ & $ 2 $ & $ 3 $ & $ 4 $ \\ \hline
$ 2 $ & $ 79 $ & $ 1 $ & $ 3 $ & $ 4 $ \\ \hline
$ 3 $ & $ 8 $ & $ 21 $ & $ 11 $ & $ 3 $ \\ \hline
$ 3 $ & $ 10 $ & $ 27 $ & $ 13 $ & $ 3 $ \\ \hline
\end{tabular}
\begin{tabular}{||c|c|c|c|c||}
\hline
$C_1$ & $C_2$ & $x$ & $y$ & $n$\\
\hline\hline
$ 3 $ & $ 17 $ & $ 6 $ & $ 5 $ & $ 3 $ \\ \hline
$ 3 $ & $ 35 $ & $ 186 $ & $ 47 $ & $ 3 $ \\ \hline
$ 3 $ & $ 43 $ & $ 10 $ & $ 7 $ & $ 3 $ \\ \hline
$ 3 $ & $ 43 $ & $ 712 $ & $ 115 $ & $ 3 $ \\ \hline
$ 3 $ & $ 73 $ & $ 72 $ & $ 25 $ & $ 3 $ \\ \hline
$ 3 $ & $ 80 $ & $ 639 $ & $ 107 $ & $ 3 $ \\ \hline
$ 5 $ & $ 1 $ & $ 4 $ & $ 3 $ & $ 4 $ \\ \hline
$ 5 $ & $ 7 $ & $ 2 $ & $ 3 $ & $ 3 $ \\ \hline
$ 5 $ & $ 14 $ & $ 37 $ & $ 19 $ & $ 3 $ \\ \hline
$ 5 $ & $ 16 $ & $ 43 $ & $ 21 $ & $ 3 $ \\ \hline
$ 5 $ & $ 22 $ & $ 1 $ & $ 3 $ & $ 3 $ \\ \hline
$ 5 $ & $ 23 $ & $ 8 $ & $ 7 $ & $ 3 $ \\ \hline
$ 5 $ & $ 61 $ & $ 2 $ & $ 3 $ & $ 4 $ \\ \hline
$ 5 $ & $ 61 $ & $ 54 $ & $ 11 $ & $ 4 $ \\ \hline
$ 5 $ & $ 61 $ & $ 326 $ & $ 27 $ & $ 4 $ \\ \hline
$ 5 $ & $ 61 $ & $ 326 $ & $ 81 $ & $ 3 $ \\ \hline
$ 5 $ & $ 76 $ & $ 1 $ & $ 3 $ & $ 4 $ \\ \hline
$ 5 $ & $ 76 $ & $ 487 $ & $ 33 $ & $ 4 $ \\ \hline
$ 6 $ & $ 1 $ & $ 20 $ & $ 7 $ & $ 4 $ \\ \hline
$ 6 $ & $ 17 $ & $ 45 $ & $ 23 $ & $ 3 $ \\ \hline
$ 6 $ & $ 19 $ & $ 51 $ & $ 25 $ & $ 3 $ \\ \hline
$ 6 $ & $ 29 $ & $ 4 $ & $ 5 $ & $ 3 $ \\ \hline
$ 6 $ & $ 29 $ & $ 185 $ & $ 59 $ & $ 3 $ \\ \hline
$ 6 $ & $ 31 $ & $ 19 $ & $ 13 $ & $ 3 $ \\ \hline
\end{tabular}
\begin{tabular}{||c|c|c|c|c||}
\hline
$C_1$ & $C_2$ & $x$ & $y$ & $n$\\
\hline\hline
$ 6 $ & $ 71 $ & $ 3 $ & $ 5 $ & $ 3 $ \\ \hline
$ 6 $ & $ 71 $ & $ 378 $ & $ 95 $ & $ 3 $ \\ \hline
$ 6 $ & $ 73 $ & $ 390 $ & $ 97 $ & $ 3 $ \\ \hline
$ 7 $ & $ 13 $ & $ 4 $ & $ 5 $ & $ 3 $ \\ \hline
$ 7 $ & $ 20 $ & $ 53 $ & $ 27 $ & $ 3 $ \\ \hline
$ 7 $ & $ 20 $ & $ 1 $ & $ 3 $ & $ 3 $ \\ \hline
$ 7 $ & $ 22 $ & $ 59 $ & $ 29 $ & $ 3 $ \\ \hline
$ 7 $ & $ 29 $ & $ 10 $ & $ 9 $ & $ 3 $ \\ \hline
$ 7 $ & $ 38 $ & $ 21 $ & $ 5 $ & $ 5 $ \\ \hline
$ 7 $ & $ 53 $ & $ 2 $ & $ 3 $ & $ 4 $ \\ \hline
$ 7 $ & $ 58 $ & $ 9 $ & $ 5 $ & $ 4 $ \\ \hline
$ 7 $ & $ 62 $ & $ 3 $ & $ 5 $ & $ 3 $ \\ \hline
$ 7 $ & $ 68 $ & $ 5 $ & $ 3 $ & $ 5 $ \\ \hline
$ 7 $ & $ 71 $ & $ 92 $ & $ 39 $ & $ 3 $ \\ \hline
$ 7 $ & $ 74 $ & $ 1 $ & $ 3 $ & $ 4 $ \\ \hline
$ 7 $ & $ 78 $ & $ 85 $ & $ 37 $ & $ 3 $ \\ \hline
$ 10 $ & $ 17 $ & $ 1 $ & $ 3 $ & $ 3 $ \\ \hline
$ 10 $ & $ 29 $ & $ 77 $ & $ 39 $ & $ 3 $ \\ \hline
$ 10 $ & $ 31 $ & $ 83 $ & $ 41 $ & $ 3 $ \\ \hline
$ 10 $ & $ 37 $ & $ 122 $ & $ 53 $ & $ 3 $ \\ \hline
$ 10 $ & $ 41 $ & $ 2 $ & $ 3 $ & $ 4 $ \\ \hline
$ 10 $ & $ 43 $ & $ 350 $ & $ 107 $ & $ 3 $ \\ \hline
$ 10 $ & $ 71 $ & $ 1 $ & $ 3 $ & $ 4 $ \\ \hline
$ 10 $ & $ 73 $ & $ 22 $ & $ 17 $ & $ 3 $ \\ \hline
\end{tabular}


\end{document}